\documentclass{amsart}
\numberwithin{equation}{section}
\usepackage{amssymb}
\usepackage{amsmath}
\usepackage{amsthm}
\usepackage{amscd}
\usepackage{amsfonts}
\usepackage{ascmac}
\usepackage[usenames]{color}
\usepackage{graphics}

\theoremstyle{plain}
\newtheorem{theorem}{Theorem}[section]
\newtheorem{proposition}[theorem]{Proposition}

\newtheorem{claim}[theorem]{Claim}
\newtheorem{example}[theorem]{Example}
\newtheorem{definition}[theorem]{Definition}

\newtheorem{fact}[theorem]{Fact}
\newtheorem{remark}[theorem]{Remark}

\newcommand{\bfC}{{\mathbb C}}
\newcommand{\bfP}{{\mathbb P}}
\newcommand{\bfR}{{\mathbb R}}
\newcommand{\bfZ}{{\mathbb Z}}

\newcommand{\bfQ}{{\mathbb Q}}

\newcommand{\barg}{{\overline g}}

\newcommand{\barj}{{\overline j}}

\newcommand{\barz}{{\overline z}}

\newcommand{\barpartial}{{\overline \partial}}

\newcommand{\txi}{{\widetilde \xi}}

\newcommand{\mapright}[1]{\smash{\mathop{   \hbox to 0.7cm{\rightarrowfill}}
  \limits^{#1}}}

\def\bp{\overline{\partial}}

\def\p{\partial}
\def\bp{\overline{\partial}}

\def\vph{\varphi}

\begin{document}

\title{Irregular Eguchi-Hanson type metrics and their soliton analogues}
\author{Akito Futaki}
\address{Yau Mathematical Sciences Center, Tsinghua University, Haidian district, Beijing 100084, China}
\email{futaki@tsinghua.edu.cn}
\date{January 31, 2021 }

\begin{abstract}
We verify the extension to the zero section of momentum construction of 
K\"ahler-Einstein metrics and K\"ahler-Ricci solitons on the total space $Y$ of positive rational powers of the 
canonical line bundle of toric Fano manifolds with possibly irregular Sasaki-Einstein metrics.
More precisely, we show that the extended metric along the zero section has an expression which 
can be extended to $Y$, 
restricts to the associated unit circle bundle as a transversely K\"ahler-Einstein 
(Sasakian eta-Einstein) metric scaled in the Reeb flow direction, and that there is a Riemannian submersion from the scaled Sasakian eta-Einstein metric to the
induced metric of the zero section.
\end{abstract}



\maketitle

\section{Introduction}
The Eguchi-Hanson metric (\cite{eguchi-hanson}, 1979) is a complete Ricci-flat K\"ahler metric on the canonical line bundle
of $\bfC\bfP^1$, also expressed as a gravitational instanton. Its holonomy group is $SU(2) = Sp(1)$, and this gives a hyperk\"ahler
structure. Around the same period Calabi (\cite{calabi79}, 1979) constructed independently a hyperk\"ahler metric on the cotangent bundle of $\bfC\bfP^m$
for $m \ge 1$. Calabi's method was to reduce obtaining a K\"ahler potential with good properties to an ordinary differential equation
when there is a large group of symmetries. This method, now called the Calabi ansatz, was applied later in many ways by many other mathematicians,
typically the momentum construction of Hwang-Singer \cite{Hwang-Singer02} 
and Feldman-Ilmanen-Knopf \cite{FIK}. 
In our papers \cite{Fut2007}, \cite{Futaki11Tohoku},  \cite{FutakiWang11}, we took up the works of Hwang-Singer
and Feldman-Ilmanen-Knopf 
to combine their ideas with the existence of Sasaki-Einstein metrics on toric Sasakian manifolds \cite{FOW}. 
Among other things we tried to show the existence of a complete Ricci-flat K\"ahler metric on the canonical line bundle $K_M$ 
of a toric Fano manifold $M$ in \cite{Fut2007}, \cite{Futaki11Tohoku}, and complete K\"ahler-Ricci solitons on some positive rational powers
of $K_M$ in \cite{FutakiWang11}. However we left open the issue of extension to the zero section of the line bundles 
when the Reeb vector field is irregular.
In this paper we verify the extension to the zero section of the momentum construction of 
K\"ahler-Einstein metrics and K\"ahler-Ricci solitons on the total space of some positive rational powers of 
canonical line bundle of toric Fano manifolds with possibly irregular Sasaki-Einstein metrics. Our results are described as follows.

\begin{theorem}\label{main2} Let $M$ be a toric Fano manifold and 
$L$ a holomorphic line bundle over $M$ such that $K_M = L^{\otimes p}$ for some positive integer $p$. 
Then for an integer $k \ge p$, there exists
a complete K\"ahler-Einstein metric $\omega_\varphi$ 
on the total space $Y$ of $L^{\otimes k}$ (denoted by $L^k$ hereafter) with $\rho_{\omega_\varphi} = (\frac{2p}k - 2)\omega_\varphi$, 
where $\rho_{\omega_\varphi}$ is the Ricci form of $\omega_\varphi$ and $\varphi$ is the profile of the momentum construction starting from the K\"ahler cone metric $\omega$
of a possibly irregular Sasakian $\eta$-Einstein (transversely K\"ahler-Einstein) metric 
on the associated unit circle bundle of $L^k$ (see Section 3 for the profile). 
The resulting metric along the zero section has an expression which 
 can be extended to $Y$, 
restricts to the associated unit circle bundle as a transversely K\"ahler-Einstein 
(Sasakian $\eta$-Einstein) metric scaled by a constant given by \eqref{C139} in the Reeb flow direction, and the 
induced metric of the zero section has a submersion from this scaled Sasakian $\eta$-Einstein metric. 

In particular, this construction gives a complete Calabi-Yau metric on the total space of $K_M$, with k=p in this case.
\end{theorem} 

Note that the associated circle bundle in the statement of Theorem \ref{main2} is strictly speaking the set $\{r=1\}$ 
in the K\"ahler cone of the Sasakian $\eta$-Einstein manifold where $r$ is the radial function.

\begin{remark}\label{rem1} The cohomology class $[\omega_{\varphi}]$ corresponds to the $\mathbb Q$-line bundle $\pi^\ast K_M^{-\frac{k}{2p}} = \pi^\ast L^{-\frac{k}2}$ under $\mathrm{Pic}(Y) = H^2(Y,\mathbb Z)$ (c.f. \cite{vanC11}), where we have put $Y$ to be the
total space of $L^{k}$ and $\pi : Y \to M$ is the projection. But this class belongs to the image of $H^2_c(Y,\bfZ) \hookrightarrow H^2(Y,\bfZ)$
so that it has a compact support.
\end{remark}

As indicated in the statement, Theorem \ref{main2} uses the momentum construction of Hwang-Singer 
\cite{Hwang-Singer02}. In the meantime after \cite{Fut2007}, the existence of Calabi-Yau metrics
on crepant resolutions of Calabi-Yau cones have been obtained by
\cite{vanC10}, \cite{vanC11}, \cite{Goto}, \cite{ConlonHein13}, \cite{ConlonHein15} by
the method of the seminal work by Joyce \cite{Joyce}, and by Biquard and Macbeth \cite{BiquardMacbeth} by the 
gluing method. 
 Their results imply that a Calabi-Yau metric exists for each K\"ahler class in the 2nd cohomology class.
 
 \begin{remark}\label{rem2}
 The Calabi-Yau metric in the case of $k=p$ in Theorem \ref{main2} is asymptotic with order $-2m$ to 
 the Calabi-Yau cone metric
 corresponding to the Sasaki-Einstein metric, where $m = \dim_\bfC M$. Thus by the uniqueness result in \cite{ConlonDeruelleSun19},
the constructions in \cite{vanC10} and \cite{Goto} recover  our Calabi-Yau metric in Theorem \ref{main2}.
 \end{remark}
 
 The K\"ahler-Einstein metrics obtained in Theorem \ref{main2} have explicit descriptions near the zero section.
 In particular, in the Calabi-Yau case, this is an extra degree of information that one obtains as a result of solving an ODE
 rather than a PDE as was obtained in \cite{Goto, vanC10, vanC11}, where information concerning the Calabi-Yau metrics in a neighborhood of the zero section is lost.

\begin{theorem}\label{main4} 
Let $M$ be a toric Fano manifold and $K_M = L^{p}$, $p \in \bfZ^+$, then 
 the total space of $L^{k}$ admits a complete expanding K\"ahler-Ricci soliton
 if $k > p$, a complete steady K\"ahler-Ricci soliton if k = p and a complete shrinking K\"ahler-Ricci soliton if $k<p$.
 These solitons are obtained by the momentum construction, 
 and the resulting metric along the zero section has an expression similar to Theorem \ref{main2}.
\end{theorem}
In this soliton case we also use Hwang-Singer's momentum construction while existence results using PDE methods
have been known by \cite{Sie13}, \cite{ConlonDeruelle20JDG} for expanding solitons and by \cite{ConlonDeruelle20}
for steady solitons. But the existence of complete shrinking solitons obtained in Theorem \ref{main4} is new. 
A notion of stability for such metrics was alluded to but not defined in \cite{ConlonDeruelleSun19}. The Calabi-Yau cones
coming from line bundles as in Theorem \ref{main4} should be stable whatever the definition of stability may be. 
For uniqueness there are works by \cite{ConlonDeruelleSun19}, \cite{Sch20} and \cite{Cifarelli}. 

The difference of the arguments between our earlier works \cite{Fut2007}, \cite{Futaki11Tohoku},  \cite{FutakiWang11}
and the present paper is as follows. In the earlier papers we tried to describe the metric near the zero section using
the complex coordinate along the complexified Reeb flow, and since it is irregular in general it appeared difficult
to use it for the description. In the present paper we describe the behavior of the metric on the level set of the radial
coordinate $r$ of the K\"ahler cone metric associated with the Sasakian $\eta$-Einstein manifold before starting the momentum construction, 
along the flow of the radial vector field $r\partial/\partial r$, and observe the metric converges as $r \to 0$ to a 
metric along the zero section. The description of this limiting metric along the zero section has a form which can be considered as defining
a metric extending over $L^k$ minus the zero section, and if we restrict this metric on the set $\{r=1\}$ it coincides with
the Sasakian $\eta$-Einstein metric
scaled by a positive constant in the (real) Reeb flow direction. The induced metric on the zero section has a 
submersion from this scaled Sasakian $\eta$-Einstein metric. 

After this introduction, in Section 2, we recall basic facts about Sasakian geometry, the volume minimization arguments 
of Martelli-Sparks-Yau \cite{MSY1}, \cite{MSY2} and the existence of Sasaki-Einstein metrics on toric Sasaki manifolds \cite{FOW}.
In Section 3 we set up the 
momentum construction for complete K\"ahler-Einstein metrics and prove Theorem \ref{main2}, Remark \ref{rem1} 
and Remark \ref{rem2}.
In Section 4 we study the soliton case and prove Theorem \ref{main4}. 

\section{Preliminaries on Sasakian geometry.} 
In this section we briefly review Sasakian geometry. The reader is referred to \cite{BGbook}, \cite{MSY2}, \cite{CFO}, \cite{futaki10} and
\cite{FO_ICCM_Notices19} for more detail and related topics.
A Sasakian manifold is by definition a Riemannian manifold $(S, g)$ whose Riemannian cone manifold 
$(C(S), \overline g)$ with $C(S) \cong  S\times \bfR^+$ and $\overline g = dr^2  + r^2g$ 
is a K\"ahler manifold, where $r$ is the standard coordinate on $\bfR^+$.
But it is important to note that $r$ is a smooth function on $C(S)$ through the identification $C(S) \cong  S\times \bfR^+$,
and a deformation of Sasakian structure is given by the deformation of the smooth function $r$ on $C(S)$.

From this definition, a Sasakian manifold $S$ has odd-dimension $\dim _{\bfR}S = 2m + 1$, and thus
$\dim_{\bfC} C(S) = m+1$.
$S$ is always identified with the
real hypersurface $\{r=1\}$ in $C(S)$ and inherits the Riemannian submanifold structure and other various structures
from $C(S)$ as described below.
Algebraically, $C(S)$ is a normal affine algebraic variety (see \cite{vanC10}, Section 3.1), and the apex of the cone is the origin which is the unique singularity.

To study the differential geometry of a Sasakian manifold, it is also important to notice that the K\"ahler form
$\overline{\omega}$ 
of the cone $C(S)$ is expressed using the radial function $r$ as
\begin{equation*}\label{K1}
\overline{\omega} = \frac{i}2 \partial\barpartial r^2.
\end{equation*}
Thus the geometry of the K\"ahler cone is determined only from $r$ and the complex structure, denoted by $J$.
Hence, with fixed $J$, the Sasakian structure is also determined by the smooth function $r$.
As noted in the first paragraph of this section, the Sasakian structure is deformed by the deformation of 
the smooth function $r$ satisfying $i \partial\barpartial r^2 > 0$. 

Since $S$ is identified with the submanifold $\{r=1\}$, the Sasakian geometry of $S$ as submanifold geometry is described only by using $r$ and the complex structure $J$.
One can convince oneself of this fact and the facts described below if one examines the standard example of the unit sphere $\{r=1\}$ in $\bfC^{m+1}$
with
$$r^2 = (|z_0|^2 + \cdots + |z_m|^2).$$
Putting $\txi = J(r\frac{\p}{\partial r})$, 
$\txi - iJ\txi$ defines a holomorphic vector field on $C(S)$. 
The restriction of $\txi$ to 
$S = \{r=1\}$, which is tangent to $S$, is called the {\it Reeb vector field} of $S$ and
denoted by $\xi$. 
The Reeb vector field $\xi$ is a Killing vector field on $S$ and 
generates a $1$-dimensional foliation $\mathcal F_{\xi}$, called the Reeb
foliation of $S$. It is also possible to consider the group of isometries generated by the flow of $\xi$, called the Reeb flow
which we denote also by $\mathcal F_{\xi}$. 
The closure of the Reeb flow is a toral subgroup 
in the isometry group of $S$. 
If the dimension of the toral group is equal to (resp. greater than) $1$ the Sasakian manifold is said to be quasi-regular (resp. irregular),
and if the Sasakian manifold is quasi-regular and the $S^1$-action is free it is said to be regular.
In the classical Sasakian geometry, the standard normalization of the metric
is chosen so that the length of $\xi$ is $1$. Below we see some clumsy coefficients e.g. \eqref{SE1}, \eqref{SE2}, but they come from 
this normalization. We will keep this normalization since it is natural as long as we adopt the above definition of Sasakian manifolds.

Let $\eta$ be the dual $1$-form
to $\xi$ using the Riemannian metric $g$, i.e. $\eta = g(\xi, \cdot)$. 
To describe $\eta$ in terms of $r$, it is convenient to introduce the operator $d^c = (i/2)(\barpartial - \p)$. 
Our choice of the factor $1/2$ is to make the equality $dd^c = i\p\barpartial$ hold. This choice was convenient
in our earlier paper \cite{FOW}, and we will continue to use it in this paper.
Then $\eta$ can be expressed as
$$\eta = (i(\barpartial - \p) \log r)|_{r=1} =  (2d^c \log r) |_{r=1}.$$
Then $d\eta$ is non-degenerate on $D := \mathrm{Ker}\, \eta$ and thus $S$ becomes a contact manifold
with the contact form $\eta$. $D$ is called the {\it contact bundle}, which is a smooth complex vector bundle over $S$
and has the first Chern class $c_1(D) \in H^2(S,\bfZ)$.
The Reeb vector field $\xi$ satisfies
$$i(\xi)\eta= 1\quad \mathrm{and}\quad i(\xi) d\eta = 0,$$
where $i(\xi)$ denotes the interior product, which are often used as the defining
properties of the Reeb vector field for contact manifolds.
The local orbit spaces of $\mathcal F_{\xi}$ admits 
a well-defined K\"ahler structure, and the pull-back of the local K\"ahler forms to $S$ 
glue together to give a global $2$-form 
$$\omega^T = \frac12 d\eta = d(d^c \log r\,|_{r=1}) = (dd^c \log r)|_{r=1}$$
on $S$, which we call the {\it transverse K\"ahler form}. 
We call the collection of K\"ahler structures on local leaf spaces of $\mathcal F_{\xi}$
the {\it transverse K\"ahler structure}. 
A smooth differential form $\alpha$ on $S$ is said to be basic if
$$ i(\xi)\alpha = 0\quad \mathrm{and}\quad \mathcal L_{\xi}\alpha = 0,$$
where $\mathcal L_{\xi}$ denotes the Lie derivative by $\xi$.
For example, the transverse K\"ahler form $\omega^T$ is a basic 2-form.
The basic forms are lifted from differential forms on local orbit spaces of $\mathcal F_\xi$, and 
preserved by the exterior derivative $d$ which decomposes into
$d = \p_B + \barpartial_B$. We can define basic cohomology groups using $d$ and
basic Dolbeault cohomology groups using $\barpartial_B$. We also have the transverse Chern-Weil theory and
can define basic Chern classes for complex vector bundles with basic transition functions.
As in the K\"ahler case, the basic first Chern class $c_1^B$
of the Reeb foliation is represented by the $1/2\pi$ times the transverse
Ricci form $\rho^T$:
\begin{equation*}
\rho^T = -i \p_B\bp_B \log \det(g^T_{i\barj}),
\end{equation*} 
where 
$$\omega^T = i \ g_{i\barj}^T\ dz^i \wedge dz^{\barj}$$
and $z^1,\ \cdots,\ z^m$ are local holomorphic coordinates on the
local orbit space of $\mathcal F_\xi$.
A Sasakian manifold $(S,g)$ is called a Sasaki-Einstein manifold if $g$ is an Einstein metric.
\begin{fact}[c.f.  \cite{BGbook}]\label{s3}Let $(S,g)$ be a $(2m+1)$-dimensional Sasakian manifold. 
The following three conditions are equivalent.
\begin{enumerate}
\item[(a)] $(S,g)$ is a Sasaki-Einstein manifold. The Einstein constant is necessarily $2m$: 
\begin{equation}\label{SE1}
\mathrm{Ric}_g = 2m g, 
\end{equation}
where $ \mathrm{Ric}_g$ denotes the Ricci curvature of $g$.
\item[(b)] $(C(S), \overline g)$ is a Ricci-flat K\"ahler manifold.
\item[(c)] The local orbit spaces of the Reeb flow have transverse K\"ahler-Einstein metrics
with Einstein constant $2m+2$:
\begin{equation}\label{SE2}
\rho^T = (2m+2) \omega^T.
\end{equation}
\end{enumerate}
\end{fact}
\noindent
One may compare Fact \ref{s3} with \eqref{s8} - \eqref{s11} below.

In the previous paragraph we defined the contact form $\eta$ and the transverse K\"ahler form $\omega$ on $S$.
For the purpose of momentum construction of this paper, it is more convenient to consider them to be
lifted to the K\"ahler cone $C(S)$ as
\begin{equation}\label{C1}
\eta = 2d^c \log r
\end{equation}
and
\begin{equation}\label{s1}
\omega^T = \frac12 d\eta = dd^c \log r.
\end{equation}
A moment's thought shows the transverse Ricci form $\rho^T$ also lifts to $C(S)$. 
If $S$ is a Sasaki-Einstein manifold then by (c) of Fact \ref{s3}, we have $c_1^B > 0$, i.e. 
$c_1^B$ is represented by a positive basic $(1,1)$-form. Moreover, 
under the natural homomorphism $H^2_B(\mathcal F_{\xi}) \to H^2(S)$ of the basic cohomology group 
$H^2_B(\mathcal F_{\xi})$ to ordinary de Rham cohomology group $H^2(S)$, the basic first Chern class $c_1^B$ is sent
to the ordinary first Chern class $c_1(D)$ of the contact distribution $D$ (see the paragraph before Fact \ref{s3})
since the expressions as de Rham classes are the same by the Chern-Weil theory. But by \eqref{SE2} and \eqref{s1},
\begin{equation*}\label{s5}
c_1(D) = (2m+2)[\omega^T] = (m+1)[d\eta] = 0.
\end{equation*}
(Notice that $\omega^T$ is a positive form as a basic form, which is a transverse K\"ahler form, 
but that $\omega^T = \frac 12 d\eta$ is an exact form as an 
ordinary 2-form on $S$.)
Conversely, if $c_1^B > 0$ and $c_1(D) = 0$, then $c_1^B = \tau [d\eta]$ for some positive
constant $\tau$ (See \cite{BGbook}, Corollary 7.5.26. The proof follows from an exact sequence known in 
the foliation theory (the diagram 7.5.11 in \cite{BGbook}).)

A Sasakian manifold 
$(S, g)$ is said to be toric if the K\"ahler cone manifold $(C(S), \barg)$ is toric, namely if there exists an 
$(m+1)$-dimensional real torus $T^{m+1}$ acting on $(C(S), \overline g)$ effectively as
holomorphic isometries fixing the apex of $C(S)$. Then $T^{m+1}$ preserves $\txi$ because $T^{m+1}$ preserves $r$ and the
complex structure $J$, and so the flow generated by $\txi$.
This latter statement implies that $[\txi, \mathrm{Lie}(T^{m+1})] = 0$. It follows that $\txi \in \mathrm{Lie}(T^{m+1})$
because $T^{m+1}$ is maximal in the isometry group of the link of the cone, and that the closure of the flow is a 
toral subgroup of $T^{m+1}$.
\begin{theorem}[\cite{FOW, CFO}]\label{s7}\ \ \\
(i)  Let $S$ be a compact toric Sasakian manifold with $c_1^B > 0$ and $c_1(D) = 0$. Then $S$ admits a possibly irregular Sasaki-Einstein metric by deforming the 
Sasakian structure by varying the Reeb vector field and then performing a transverse
K\"ahler deformation.\\
(ii) For a compact toric Sasakian manifold $S$, the conditions $c_1^B > 0$ and  $c_1(D)=0$ is equivalent to the $\bfQ$-Gorenstein property of K\"ahler cone $C(S)$.
\end{theorem}
\noindent
The proof of Theorem \ref{s7} is outlined as follows.
To prove (i), start with an arbitrary Sasakian structure, we wish to deform the Sasakian structure to obtain a new structure
with vanishing obstruction (as in \cite{futaki83.1}) for the existence of K\"ahler-Einstein metric in terms of the transverse K\"ahler structure. 
The idea is to use the volume minimization of Martelli-Sparks-Yau \cite{MSY1}, \cite{MSY2}. 
There are two approaches to describe this. The first approach is to consider the deformation of
K\"ahler structures on the cone $C(S)$. 
As explained at the beginning of this section, the deformation of K\"ahler cone structures are obtained by 
variations of the radial function $r$, so that a variation of the Sasakian structure is described in the form 
$r^{\prime 2}= r^2 e^\psi$ for a smooth function $\psi$ on $C(S)$ with $i\partial\barpartial\, r^{\prime 2} > 0$. We consider the volume functional
on the space of all Sasakian structures on $S$ with fixed complex structure $J$, or equivalently all K\"ahler cone structures on $C(S)$, by assigning the volume $\mathrm{Vol}(\{r=1\})$ of the Sasakian manifold $\{r=1\}$. The first and the second variation formulas
were given in \cite{MSY2}, Appendix C, and it was shown in \cite{FOW}, Proposition 8.7, in the general setting (including non-toric case) that if the first variation vanishes then the obstruction for the existence of transverse K\"ahler-Einstein metric vanishes. When we consider toric Sasakian manifolds and toric K\"ahler cone manifolds we just consider 
the $T^{m+1}$-invariant radial functions so that $\psi$ above is also $T^{m+1}$-invariant. 

The second approach is specific to the toric case. Note that $\mathrm{Vol}(\{r=1\})$ coincides up to a universal 
constant with the volume of the intersection of
the moment map image of the K\"ahler cone, which we call the moment cone and denote by $\mathcal C$, and a hyperplane $H$ (see below for the exact
expression of $H$). This intersection $\mathcal C \cap H$ is the moment map image (in the sense of contact geometry) of the Sasakian manifold determined by $r$. Since $\mathcal C$ is a convex polyhedral cone, $\mathcal C \cap H$ is a 
convex polyhedral compact set. A salient feature of toric K\"ahler geometry is the existence of a distinguished point $\beta$ in $\mathcal C$ 
such that
$$ \langle\lambda_j, \beta\rangle = 1,$$
where $\lambda_j \in \mathrm{Lie}(T^{m+1})$, $j = 1, \cdots, d$, are the normal vectors to the facets of $\mathcal C$;
$$ \mathcal C = \{y \in (\mathrm{Lie}(T^{m+1}))^\ast\,|\, \langle\lambda_j, y\rangle \ge 0,\ j= 1, \cdots, d\}.$$
(In the notation of \cite{MSY1}, $\beta = -\gamma$.)
A differential geometric derivation of $\beta$ using the standard formalism by Delzant \cite{Delzant} and Guillemin \cite{Guillemin94} can be found in \cite{MSY1}, Section 2, but an algebro-geometric derivation in toric geometry is also known, see e.g. de Borbon-Legendre \cite{deBorbonLegendre}, Section 4.2. Moreover, in Theorem 1.2 of their paper 
\cite{deBorbonLegendre}, the role of $\beta$ is clarified in terms of the cone angles along the divisors corresponding to the boundary facets of $\mathcal C$ when the conditions
in (ii), Theorem \ref{s7} of this paper are not satisfied. Further, in \cite{MSY1}, Section 2, it is shown that the Reeb vector field $\xi \in \mathrm{Lie}(T^{m+1})$ satisfies
$$ \langle \xi, \beta \rangle = m+1.$$
Returning to the volume minimization, we take the deformation space of Reeb vector fields, which is considered as the deformation space of toric Sasakian structures, 
to be the hyperplane $\Xi_\beta$ in the dual cone 
$$\mathcal C^\ast = \{x \in (\mathrm{Lie}(T^{m+1})\,|\, \langle x, y\rangle \ge 0 \text{ for all } y \in \mathcal C\}$$
given by
$$ \Xi_\beta := \{ \xi^\prime \in \mathcal C^\ast\,|\, \langle\xi^\prime, \beta\rangle = m+1\}.$$
Note that $\xi$ is contained in $\Xi_\beta$. 
Consider the volume functional on $\Xi_\beta$ given by 
$$ \mathrm{Vol}(\xi^\prime) := \mathrm{Vol}(P_{\xi^\prime}), $$
where $P_{\xi^\prime} = \{y \in \mathcal C\,|\, \langle\xi^\prime, y\rangle = m+1\}$. 
Note that $P_{\xi^\prime}$ passes through $\beta$ for any $\xi^\prime \in \Xi_\beta$.
The the intersection $\mathcal C \cap H$ is $(1/2(m+1))P_\xi$, see (2.68) in \cite{MSY1} for a proof. 
Then this volume functional on $\Xi_\beta$ turns out to be proper and convex, and thus have a 
unique critical point. The critical point $\xi^\prime$ is exactly when $\beta$ is the barycenter of $P_\xi^\prime$, and 
then one can show using Donaldson's expression of the obstruction (\cite{donaldson02}) that the obstruction vanishes for the Sasakian structure 
corresponding to $\xi^\prime$, see \cite{deBorbonLegendre}
for the detail (including the cone angle case). Then we can solve the transverse Monge-Amp\`ere equation by changing the transverse K\"ahler metric, and get a transverse
K\"ahler-Einstein metric using the analysis of Wang-Zhu \cite{Wang-Zhu04}, also \cite{donaldson0803} is recommended for the estimates. Hence by Fact \ref{s3}, we obtain a Sasaki-Einstein metric. The critical Reeb vector field 
$\xi^\prime$ is possibly irregular. These are the outline of the proof of Theorem \ref{s7}, (i). 

Theorem \ref{s7}, (ii), can be found in \cite{CFO}, Theorem 1.2. It essentially follows from the equation (23) in \cite{CFO} and the rationality of $\gamma = -\beta$.

\begin{remark}\label{rem3}Let $r$ and $r^\prime$ be the radial functions of the Sasakian structures corresponding to $\xi$ and $\xi^\prime$ in the
arguments in the previous paragraph where $\xi^\prime$ is the critical point of the volume functional. Then $(1/(2(m+1))P_\xi$ and 
$(1/(2(m+1))P_{\xi^\prime}$ are the moment map images of $\{r=1\}$ and $\{r^\prime = 1\}$ as described above. 
From this, the description of the deformation of Sasakian structure in terms of the radial functions is given by
\begin{equation}\label{rem4}
r^{\prime 2} = r^2 \exp \psi
\end{equation}
for a $T^{m+1}$-invariant smooth function $\psi$ on $\{r = 1\} \cong S$. 
Indeed, $\exp(\frac12\psi )= \langle\xi^\prime, y\rangle / \langle\xi, y\rangle$, see (2.68) in \cite{MSY1}.
\end{remark}

A Sasakian metric $g$ is said to be {\it $\eta$-Einstein} if there exist constants $\lambda$ and
$\nu$ such that 
\begin{equation}\label{s8}
\mathrm{Ric}_g = \lambda\, g + \nu\, \eta\otimes \eta.
\end{equation}
By elementary computations in Sasakian geometry, 
we always have $\mathrm{Ric}_g(\xi,\xi) = 2m$ on any Sasakian manifolds.
This implies that
\begin{equation}\label{s9}
\lambda + \nu = 2m.
\end{equation}
In particular, $\lambda = 2m$ and $\nu = 0$ for a Sasaki-Einstein metric.

Let $\mathrm{Ric}^T$ denote the Ricci curvature of the local orbit space of $\mathcal{F}_\xi$. Then again elementary
computations show
\begin{equation}\label{s10}
\mathrm{Ric}_{g} = \mathrm{Ric}^{T} - 2g^{T} + 2m \eta \otimes \eta
\end{equation}
and that the condition of being an $\eta$-Einstein metric is equivalent to
\begin{equation}\label{s11}
\mathrm{Ric}^T = (\lambda + 2)g^T.
\end{equation}

Given a Sasakian manifold with the K\"ahler cone metric $\barg = dr^2 + r^2 g$, we transform
the Sasakian structure by deforming $r$ into $r' = r^a$ for a positive constant $a$. 
This transformation is called the $D$-homothetic transformation. Then the new
Sasakian structure has 
\begin{equation}\label{s12}
\eta' = d \log r^a = a\eta, \quad \xi' = \frac 1a \xi,
\end{equation}
\begin{equation*}\label{s13}
g' = a g^T + a\eta \otimes a\eta = ag + (a^2 -a)\eta\otimes\eta.
\end{equation*}
Suppose that $g$ is $\eta$-Einstein with $\mathrm{Ric}_g = \lambda g + \nu \eta \otimes \eta$. 
Since the Ricci curvature of a K\"ahler manifold is invariant under homotheties, 
we have $\mathrm{Ric}^{\prime T} = \mathrm{Ric}^T$. From this and $\mathrm{Ric}_{g'}(\xi',\xi') = 2m$, we have
\begin{eqnarray*}
\mathrm{Ric}_{g'} &=& \mathrm{Ric}^{\prime T} - 2g^{\prime T} + 2m \eta' \otimes \eta' \\
&=& \lambda g^T  + 2g^T - 2a g^T + 2m \eta' \otimes \eta'. \nonumber
\end{eqnarray*}
This shows that $g'$ is $\eta$-Einstein with
\begin{equation*}\label{s15}
\lambda' +2 =  \frac{\lambda + 2 }a.
\end{equation*}
In summary, under the $D$-homothetic transformation of an $\eta$-Einstein metric $g$ with \eqref{s8}
we have
a new $\eta$-Einstein metric $g'$ with
\begin{equation}\label{s16}
\rho^{\prime T} = \rho^T, \quad \omega^{\prime T} = a\omega^T, \quad
\rho^{\prime T} = \left(\lambda' + 2\right) \omega^{\prime T} = \frac{\lambda + 2}a \omega^{\prime T},
\end{equation}
and thus, for any positive constants $\kappa$ and $\kappa'$, a transverse K\"ahler-Einstein metric
with Einstein constant $\kappa$ can be transformed by a $D$-homothetic transformation to a
transverse K\"ahler-Einstein metric with Einstein constant $\kappa'$.
{\it In particular, if we are given a Sasaki-Einstein metric $g$ with $\lambda = 2m$, we may obtain by D-homothetic transformation an $\eta$-Einstein
metric $g'$ with arbitrary $\lambda' + 2 > 0$. Conversely, if we have an $\eta$-Einstein metric with $\lambda + 2 > 0$ then we obtain a Sasaki-Einstein
metric with $\lambda' = 2m$ by $D$-homothetic transformation.}

\section{Momentum construction for Sasakian $\eta$-Einstein manifolds}

Based on the arguments on $D$-homothetic transformation in the previous section we start with a Sasakian $\eta$-Einstein manifold 
$(S, g)$ with 
$\mathrm{Ric}_g = \lambda\,g + \nu\,\eta \otimes \eta$, $\lambda + 2 > 0$, 
and with K\"ahler cone metric on $C(S)$
$$ \barg = dr^2 + r^2g.$$
Let $\omega^T = \frac12 d\eta$ be the transverse
K\"ahler form which gives positive K\"ahler-Einstein metrics on local leaf spaces with
\begin{equation}\label{C0}
 \rho^T = \kappa \omega^T,
 \end{equation}
where we have set 
$$\kappa := \lambda + 2 > 0.$$
Working on $C(S)$, we lift $\eta$ on $S$ to $C(S)$ by \eqref{C1}, 
and use the same notation $\eta$ for the lifted one to $C(S)$. Then $\omega^T$ is also lifted to
$C(S)$ by \eqref{s1}, 
and again use the same notation $\omega^T$ for the lifted one to $C(S)$. The {\it momentum construction} (or {\it Calabi ansatz})
searches for a K\"ahler form on $C(S)$ of the form
\begin{equation}\label{C3}
\omega = \omega^T + i \p\bp\, F(t),
\end{equation}
where $t = \log r$ and $F$ is a smooth function of one variable on $(t_1, t_2) \subset (-\infty, \infty)$.

We set
\begin{equation}\label{C4}
\tau = F'(t),
\end{equation}
\begin{equation}\label{C5}
\varphi(\tau) = F''(t).
\end{equation}
Since we require $\omega$ to be a positive form and
\begin{eqnarray*}
i\p\bp\, F(t) &=& i\,F''(t)\, \p t \wedge \bp t \,+ i\, F'(t)\, \p\bp t \\
&=& i\,\varphi(\tau)\, \p t\wedge \bp t \,+ \tau\, \omega^T. \nonumber
\end{eqnarray*}
then we must have $\varphi(\tau) > 0$ and $1+\tau > 0$.
We  further impose that the image of $F'$ is an open interval $(0, b)$ with $b \le \infty$, i.e.
\begin{equation*}\label{C9}
\lim_{t \to t_1^+} F'(t) = 0, \qquad \lim_{t\to t_2^-} F'(t) = b.
\end{equation*}
It follows from $\varphi(\tau) > 0$ that $F'$ is a diffeomorphism from $(t_1,t_2)$ to $(0,b)$,
so we can consider $F'$ as a coordinate change from $t$ to $\tau$. We will set up an ODE to solve
constant scalar curvature or K\"ahler-Ricci soliton equations in terms of $\varphi(\tau)$ with the new coordinate $\tau$. 
In \cite{Hwang-Singer02}, $\varphi(\tau)$ is called the {\it profile} of the momentum construction (\ref{C3}).
In the regular Sasakian case, the Reeb vector field generates a $\bfC^\ast$-action. Then $F$ is a K\"ahler 
potential along the $\bfC^\ast$-orbits and $F^\prime$ is the moment map for the $S^1$-action.

Given a positive function $\varphi > 0$ on $(0,b)$ 
such that 
\begin{equation*}\label{C11}
\lim_{\tau \to 0^+} \int_{\tau_0}^{\tau} \frac{dx}{\varphi(x)} = t_1, \qquad
\lim_{\tau \to b^-} \int_{\tau_0}^{\tau} \frac{dx}{\varphi(x)} = t_2
\end{equation*}
we can recover the momentum construction
as follows. 
Fix $\tau_0 \in (0,b)$ arbitrarily, and introduce a function $\tau(t)$ by
\begin{equation}\label{C8}
t = \int_{\tau_0}^{\tau(t)} \frac{dx}{\varphi(x)},
\end{equation}
and then $F(t)$ by
\begin{equation*}\label{C10}
F(t) = \int_{\tau_0}^{\tau(t)} \frac{xdx}{\varphi(x)}.
\end{equation*}
Then $F$ and $\varphi$ satisfy \eqref{C4} and \eqref{C5}, and thus
\begin{eqnarray}\label{C13}
\omega_{\varphi} &:=& \omega^T + dd^c\,F(t) \nonumber\\
&=& (1+ \tau)\, \omega^T + \varphi(\tau)\,i \p t\wedge \bp t  \\
&=& (1+ \tau)\, \omega^T + \varphi(\tau)^{-1}\,i \p \tau\wedge \bp \tau. \nonumber
\end{eqnarray}
As we assume $\varphi > 0$ on $(0,b)$, $\omega_{\varphi}$ defines a K\"ahler form
and have recovered momentum construction.

Next we compute the Ricci form $\rho_{\varphi}$ and the scalar curvature $\sigma_{\varphi}$ of 
$\omega_{\varphi}$. 
If we choose $z^0$ to be the coordinate along the holomorphic Reeb flow, 
then it is easy to check that
\begin{equation*}\label{C16}
idz^0 \wedge d\barz^0 = 2 \frac{dr}r \wedge \eta.
\end{equation*}
Using this one can compute the volume form as
\begin{equation*}\label{C17}
\omega_{\varphi}^{m+1} 
=  (1+\tau)^m (m+1) \varphi(\tau)\,\frac i2 dz^0\wedge d\barz^0\wedge(\omega^T)^m.
\end{equation*}
The Ricci form can be computed as
\begin{eqnarray*}
\rho_{\varphi} &=& \rho^T - i\p\bp \log ((1+\tau)^m \varphi(\tau)) \\
&=& \kappa \omega^T - i\p\bp \log ((1+\tau)^m \varphi(\tau)).\nonumber
\end{eqnarray*}
Using 
\begin{equation}\label{C20}
dd^c\, u(\tau)
= u'(\tau) \varphi(\tau)dd^c t + \frac 1{\varphi}(u'\varphi)'  d\tau \wedge d^c\tau
\end{equation}
for any smooth function $u$ of $\tau$, one computes
\begin{equation}\label{P2}
\rho_{\varphi} = \left(\kappa - \frac{m\varphi + (1+\tau)\varphi'}{1+\tau}\right) \omega^T - 
\left(\left(\frac{m\varphi}{1+\tau}\right)' + \varphi''\right)\varphi\,dt\wedge d^ct.
\end{equation}
From this and \eqref{C13} we see that  $\rho_{\varphi} = \alpha \omega_{\varphi}$
if and only if
\begin{equation}\label{P3}
\kappa - \frac{m\varphi + (1+\tau)\varphi'}{1+\tau} = \alpha(1+\tau),
\end{equation}
\begin{equation}\label{P4}
- (\frac{m\varphi}{1+\tau} + \varphi')' = \alpha.
\end{equation}
But \eqref{P4} follows from \eqref{P3}.

\begin{proposition}\label{sol}
Under the condition $\varphi(0) = 0$, the solution $\varphi$ of the ODE \eqref{P3} is given by
\begin{equation}\label{k4}
\varphi(\tau)= \frac{\kappa}{m+1} \left(1+\tau - \frac1{(1+\tau)^m}\right) - \frac\alpha{m+2}\left((1+\tau)^2 - \frac1{(1+\tau)^m}\right) .
\end{equation}
\end{proposition}
\begin{proof}
The ODE \eqref{P3} is equivalent to
\begin{equation*}\label{k3}
(\varphi(1+\tau)^m)' = \kappa(1+\tau)^m - \alpha (1+\tau)^{m+1}.
\end{equation*}
Using $\varphi(0) = 0$ we obtain the solution \eqref{k4}. 
\end{proof}




Now we consider the situation under which this paper considers.
Let $(L,h)$ be a negative Hermitian line bundle over a K\"ahler manifold such that
the K\"ahler form $\omega_M$ is equal to $\frac i{2\pi}\p\bp \log h$. 
Let $S_0$ be the unit circle bundle with the induced regular Sasakian structure and the radial function
$ r_0 = h(z,z)^{1/2} $
on its K\"ahler cone $C(S_0) \cong \{z \ne 0\ |\ z \in L\}$.
We consider the total space of $L$ to be a resolution of $\overline{C(S_0)}$. 
\begin{definition}\label{bundle}
A Sasakian manifold $S$ is said to be $S^1$-bundle-addapted
if the Sasakian structure of $S$ is a deformation of a regular Sasakian structure on $S_0$ 
for some negative Hermitian line bundle $\pi : (L,h) \to M$ as above such that
\begin{equation}\label{r}
r^2 = r_0^2 \exp(\pi^\ast \psi)
\end{equation}
for some smooth function
$\psi$ on $M$. 
Note that $S$ is identified with $\{r=1\}$ while $S_0$ is identified with $\{r_0 = 1\}$.
\end{definition}

As recalled in Theorem \ref{s7}, it is shown in \cite{FOW} that on a compact toric Sasakian manifold $S$ with 
positive transverse first Chern class and with 
$\bfQ$-Gorenstein K\"ahler cone $C(S)$, we can find a Sasaki-Einstein metric by varying the Reeb vector field. This
change of Reeb vector field results in a change of the radial function to the form \eqref{rem4} for a $T^{m+1}$-invariant function $\psi$. But when $S$ is a deformation of a regular Sasakian manifold $S_0$ which is the unite circle bundle
of a negative line bundle $L$ as in the setting of Theorem \ref{main2}, the function $\psi$ is in particular invariant under
the $S^1$-action, and thus descends to the base space $M$ of $L$ and is of the form $\pi^\ast\psi$ for some smooth function $\psi$ on $M$ as in \eqref{r}. 
The following proposition can be applied for this reason when $L \to M$ is a 
positive rational power of the canonical line bundle over toric Fano manifold $M$.

Note that 
the condition $\varphi(0) = 0$ in Proposition \ref{sol} is natural because in the regular Sasaki-Einstein case this condition
implies $\lim_{\tau \to 0} \omega_\varphi = \omega^T$ by \eqref{C13} and thus the solution $\omega_\varphi$ (if exists) restricts to the K\"ahler-Einstein metric on the
zero section. Thus we impose the condition $\lim_{\tau \to 0}\varphi(\tau) = 0$ hereafter.

\begin{proposition}\label{C30}Let $\omega_{\varphi}$ be the K\"ahler form obtained by
the momentum construction as above starting from a compact toric Sasakian manifold $S$ with 
an $S^1$-bundle-adapted toric $\eta$-Einstein metric $g$, and with $(t_1,t_2) = (-\infty,\infty)$ and
suppose that the profile $\varphi$ is defined on $(0,b) = (0,\infty)$ and that $\lim_{\tau \to 0}\varphi(\tau) = 0$.
Then $\omega_{\varphi}$ defines a 
complete metric, has a noncompact end towards $\tau = \infty$
and extends to a smooth metric on the total space of the line bundle up to the zero section
if and only if $\varphi$ grows at most quadratically as $\tau \to \infty$ and
$\varphi'(0) = 2$. This last condition is equivalent to $\alpha = \kappa -2 = \lambda$. 
The resulting metric along the zero section has an expression which 
 can be extended to the total space of $L$, 
restricts to the associated unit circle bundle $S \cong \{r=1\}$ as a transversely K\"ahler-Einstein 
(Sasakian $\eta$-Einstein) metric scaled by a constant given by \eqref{C139} in the Reeb flow direction, and 
there is a Riemannian submersion from the scaled Sasakian $\eta$-Einstein metric to the
induced metric of the zero section.
\end{proposition}
\begin{proof} By Proposition 3.2 in \cite{Futaki11Tohoku}, 
$\varphi$ must grow at most quadratically as $\tau \to \infty$. 
Now let us consider \eqref{C13} when $\tau \to 0$. Obviously $(1+\tau)\omega_\varphi > 0$
for $\tau \ge 0$. The second term on the right hand side of \eqref{C13} is computed as
\begin{eqnarray*}
\varphi(\tau)\, i\p t\wedge \bp t &=& \varphi(\tau) i \p\log r \wedge \bp\log r \nonumber\\
&=& \frac{\varphi(\tau)}{r^2}  i \p r \wedge \bp r.
\end{eqnarray*}
We wish to find the condition for 
$\lim_{\tau \to 0} \varphi(\tau)/r^2 $
to exist and be non-zero. Since we assume $\varphi(0) = 0$, then by Proposition \ref{sol},
$\varphi(\tau)$ is of the form 
\begin{equation*}\label{C32}
\varphi(\tau) = a_1\tau + O(\tau^2).
\end{equation*}
Since $t = \log r$, $\tau = F'(t)$ and $\varphi(\tau) = F''(t)$ we have
\begin{equation*}\label{C33}
\frac{d\tau}{dt} =\varphi(\tau) = a_1\tau + O(\tau^2).
\end{equation*}
Thus
\begin{equation*}\label{C134}
\lim_{\tau \to 0}\frac{\varphi(\tau)}{r^2} =
\lim_{t \to -\infty}\frac{\varphi'(\tau)\frac{d\tau}{dt}}{2r\frac{dr}{dt}} 
= \frac{a_1}2 \lim_{\tau \to 0}\frac{\varphi(\tau)}{r^2}.
\end{equation*}
Therefore if $\lim_{\tau \to 0} \varphi(\tau)/r^2$ exists and is non-zero then $a_1 = 2$,
i.e. $\varphi'(0) = 2$.
Conversely if $\varphi'(0) = 2$ then we have
\begin{equation*}\label{C135}
\frac{d\tau}{dt} = \varphi(\tau) = 2\tau + O(\tau^2) = 2\tau\beta(\tau)
\end{equation*}
where $\beta(\tau)$ is a function of $\tau$ real analytic near $\tau = 0$ with $\beta(0) = 1$.
Note that the real analyticity of $\beta$ follows from the real analyticity of $\varphi$ in Proposition \ref{sol}. We then have
\begin{equation*}\label{C136}
\frac{d\tau}{\tau\beta(\tau)} = 2dt
\end{equation*}
and from this
\begin{equation*}\label{C137}
\log \tau + \gamma(\tau) = c_0 + 2t
\end{equation*}
for some real analytic function $\gamma(\tau)$ of $\tau$ with $\gamma(0) = 0$ and some constant $c_0$. From this we have
\begin{equation*}\label{C138}
\tau = e^{-\gamma(\tau)} e^{c_0 + 2t} = r^2e^{c_0 - \gamma(\tau)}.
\end{equation*}
Thus we obtain
\begin{equation}\label{C139}
\frac12\lim_{\tau \to 0} \frac{\varphi(\tau)}{r^2} = \frac12\lim_{\tau \to 0}\frac{2\tau + O(\tau^2)}{r^2} = e^{c_0}.
\end{equation}
It follows from \eqref{P3} that, under the condition $\varphi(0) = 0$, the condition $\varphi^\prime(0) = 2$ is equivalent to
\begin{equation}\label{k1}
\alpha = \kappa -2 = \lambda.
\end{equation}

Thus, the limit of $\omega_\varphi$ as $\tau \to 0$ is expressed as
\begin{equation}\label{lim1}
 \omega_{\varphi} = \omega^T + 2e^{c_0} \,i \p r \wedge \bp r 
\end{equation}
restricted to $r=0$. 
Using \eqref{r},  the right hand side of \eqref{lim1} is expressed along $\{r = 0\} = \{z = 0\}$ as
\begin{equation}\label{lim2}
 \lim_{\tau \to 0} \omega_{\varphi} = \frac i2 \p\barpartial (\log h + \psi)  + e^{c_0} \,\frac i2 hdz \wedge d\barz.
\end{equation}
This is an expression of the limiting K\"ahler form on the total space of $L$ along the zero section. The K\"ahler form
of the induced metric on the zero section is given by the first term of the right hand side of \eqref{lim2}.
Let us study the geometry of the metric given by the right hand side of \eqref{lim2}, and see the induced metric on the
zero section is indeed positive definite. Let us consider the metric \eqref{C13} along $\{r = \epsilon\}$ 
\begin{eqnarray}\label{lim3}
\left.\omega_{\varphi}\right|_{r = \epsilon}&:=& (1+ \tau(\epsilon))\, \omega^T + \varphi(\tau(\epsilon))\,\left.i \p t\wedge \bp t \right|_{r=\epsilon} \nonumber\\
&=& (1+ \tau(\epsilon))\, \omega^T + \frac{\varphi(\tau(\epsilon))}{\epsilon^2}\,i \p r \wedge \bp r, 
\end{eqnarray}
and then consider the induced Riemannian metric $g_\epsilon$ to the submanifold $\{r=\epsilon\}$ which is $S^1$-equivariantly
diffeomorphic to the unit circle bundle of $L$. Since all the construction so far are $T^{m+1}$-invariant then the metric $g_\epsilon$ 
is $S^1$-invariant. It follows that there is a Riemannian submersion of $(\{r=\epsilon\}, g_\epsilon)$ to the zero
section $Z$ with some Riemannian metric which we denote by $\gamma_\epsilon$. We wish to show that the limit of $\gamma_\epsilon$ converges
to a positive definite metric $\gamma_0$. If this is confirmed then $(\{r=\epsilon\}, g_\epsilon)$ collapses to $(Z, \gamma_0)$ and the
fundamental 2-form of $\gamma_0$ is the second term of the right hand side of \eqref{lim2}.
In order to see the behavior of  submersion $(\{r=\epsilon\}, g_\epsilon) \to (Z, \gamma_\epsilon)$, we regard the right hand side of 
\eqref{lim3} as a K\"ahler metric on $L-Z$, and then restrict it to $\{r=1\}$. Let us denote this restricted metric to $\{r=1\}$ by $g_{1,\epsilon}$.
Because of the $S^1$-adaptedness, $(\{r=\epsilon\}, g_\epsilon)$ and $(\{r=1\},g_{1,\epsilon})$ only differ by scaling of the $S^1$-orbits,
and thus have the same transverse metric for the $S^1$ orbits. Hence we have a Riemannian submersion $(\{r=1\}, g_{1,\epsilon}) \to (Z,\gamma_\epsilon)$.
But $g_{1,\epsilon}$ converges as $\epsilon \to 0$ to the induced metric to $\{r=1\}$ of the metric on $L-Z$ expressed by \eqref{lim1}.
Let us put $g_{1,0} := \lim_{\epsilon \to 0}g_{1,\epsilon}$. Then there is a submersion $(\{r=1\}, g_{1,0})$ to $(Z, \gamma_0)$ for some
Riemannian metric $\gamma_0$, which is of course positive definite. It follows that $(\{r=\epsilon\}, g_\epsilon)$ collapses to $(Z, \gamma_0)$. 
Note that $g_{1,0}$ is an $\eta$-Einstein metric scaled by $e^{c_0}$ in the Reeb flow direction since
$$2i \p r \wedge\bp r = 2dr \wedge d^c r|_{r=1} = dr\wedge \eta,$$
and the right hand side is the fundamental 2-form of $dr^2 + (dr\circ J)^2$. 
This completes the proof of
Proposition \ref{C30}. 
\end{proof}
\noindent
In principle one can compute $e^{c_0}$ in \eqref{C139} for each case when the ODE is solved, see Example \ref{example}.

\begin{proof}[Proof of Theorem \ref{main2}.] For the standard regular Sasakian structure on the total space of $S^1$-bundle of $L^k = K_M^{k/p}$, 
the basic cohomology classes on $S$ are the ordinary cohomology classes of the base manifold $M$.
Since $ 2[\omega_0^T]_B  = -c_1(L^k) = -\frac kp c_1(K_M)$ for the transverse K\"ahler form $\omega_0^T$ 
of the regular Sasakian structure, we have
\begin{equation*}\label{main2.8}
[\rho_{\omega_0^T}]_B = \frac{2p}k [\omega_0^T]_B,
\end{equation*}
where $\rho_{\omega_0^T}$ denotes the Ricci form of $\omega_0^T$. 
By $D$-homothetic transformation \eqref{s12} with 
$$ a = \frac{p}{k(m+1)} $$
we obtain a Sasakian structure $g'$ with 
$$ [\rho_{\omega^{\prime T}}]_B = (2m+2) [\omega^{\prime T}]_B.$$
Since the standard Sasakian structure is toric, the volume minimizing argument 
(c.f. Theorem \ref{s7}) gives
a Sasaki-Einstein metric, still denoted by $g^\prime$, with 
$$ 
\rho_{\omega^{\prime T}} = (2m+2) \omega^{\prime T}.
$$
Then using the $D$-homothetic transformation with $a = \frac{k(m+1)}{p}$
we obtain an $\eta$-Einstein metric $g$ with
\begin{equation*}\label{main2.9}
\rho^T = \frac{2p}k \omega^T
\end{equation*}
on the total space of $S^1$-bundle of $L^{\otimes k}$. 
Here, recall that we had used the notation $\rho^T$ for $\rho _{\omega^T}$.
As discussed in the paragraph after Definition \ref{bundle}, the Sasakian $\eta$-Einstein metric
thus obtained is $S^1$-bundle-adapted 
with \eqref{r}. 
Now we start the momentum construction with this $S^1$-bundle-adapted 
Sasakian $\eta$-Einstein metric, and then obtain the ODE \eqref{P3}.
From \eqref{C0} we have $\kappa = \frac{2p}k$. 
Here, we choose $\alpha$ so that 
\begin{equation*}\label{k5}
\alpha =  \frac{2p}k - 2.
\end{equation*}
Then $\alpha = \kappa - 2 = \lambda$, which is equivalent to $\varphi'(0) = 2$.
The solution $\varphi$ with \eqref{k4} is positive for $\tau > 0$ since we assume $k \ge p$,
grows at most quadratically as $\tau \to \infty$, and satisfies $\varphi(0) = 0$ and $\varphi'(0) = 2$.
Hence it follows from Proposition \ref{C30} that 
there exists a complete K\"ahler-Einstein metric $\omega_\varphi$ 
on the total space of $L^{\otimes k}$ with 
\begin{equation*}\label{main2.7}
\rho_{\omega_\varphi} = \left(\frac{2p}k - 2\right)\omega_\varphi,
\end{equation*}
where $\varphi$ is the profile of the momentum construction.
As described in Proposition \ref{C30} and its proof, the resulting metric along the zero section has an expression which 
 can be extended to the total space $Y$ of $L^k$, 
restricts to the associated unit circle bundle $S \cong \{r=1\}$ as a Sasakian $\eta$-Einstein metric scaled by a constant given by \eqref{C139} in the Reeb flow direction, and 
there is a Riemannian submersion from the scaled Sasakian $\eta$-Einstein metric to the
induced metric of the zero section.

This metric is Ricci-flat when $k = p$, that is $L^{\otimes k} = K_M$.
This completes the proof of Theorem \ref{main2}.
\end{proof}

\begin{proof}[Proof of Remark \ref{rem1}.]
Recall that we lifted $\omega^T$ to $C(Y) = Y - Z$, where $Z$ is the zero section, by using the expression \eqref{s1}. Using it we considered
$\omega_\varphi$ by the equations \eqref{C3}. 
We showed that $\omega_\varphi$ extends to the zero section $Z$ to give a 
complete K\"ahler-Einstein metric on $Y$, and we still denote the extended metric
by the same letter $\omega_\varphi$. Thus $\omega_\varphi$ is an exact form outside the zero section $Z$.
It follows that the cohomology class $[\omega_\varphi]$ has a compact support, and in fact the support is 
on the zero section.
But \eqref{lim2} shows that the pull back of the class $[\omega_\varphi]$ to $Z$ is $- \frac12 c_1(L^k) = -\frac k{2p} c_1(K_M)$. 
This completes the proof of Remark 1.
\end{proof}
\begin{example}\label{example}
For the Calabi-Yau case $k=p$ in Theorem \ref{main2}
we have $e^{c_0} = \frac1{m+1}$, and the induced metric to the zero section of 
the irregular Eguchi-Hanson type metric has a Riemannian submersion form
the Sasakian $\eta$-Einstein metric. 
scaled by $\frac1{m+1}$ in the Reeb flow direction.
\end{example}
\begin{proof} From $k=p$ we have $\kappa =2$ and $\alpha = 0$. Thus by \eqref{k4}
\begin{equation*}
\varphi(\tau) = \frac 2{m+1} ((1+ \tau) - \frac 1{(1+\tau)^m}).
\end{equation*}
Take $\tau_0 = 2^{1/(m+1)} -1$. Then $\tau(t)$ is obtained from \eqref{C8} as
$$
t = \frac 12 \log ((\tau(t) + 1)^{m+1} -1).
$$
Since $t = \log r$, we have 
\begin{equation}\label{CY3}
\tau(t) = (r^2 + 1)^{1/(m+1)} - 1,
\end{equation}
and 
\begin{equation}\label{CY1}
\varphi(\tau) = \frac 2{m+1} \left(\left(r^2 + 1\right)^{\frac1{m+1}} - \frac 1{\left(r^2 + 1\right)^\frac m{m+1}}\right).
\end{equation}
It follows that 
$$ e^{c_0} =\frac12 \lim_{r \to 0} \frac{\varphi(\tau)}{r^2} = \frac1{m+1}.
$$
\end{proof}

\begin{proof}[Proof of Remark \ref{rem2}.]
The Calabi-Yau metric on $K_M^{-1}$ in Theorem \ref{main2} for $k=p$ is obtained by
the momentum construction starting with the Sasakian $\eta$-Einstein metric such that 
$\rho^T = 2 \omega^T$. This equality is \eqref{C0} with $\kappa = 2$. The corresponding cone
metric on $C(S)$ is not Calabi-Yau. 
In order to get a Calabi-Yau cone metric on $C(S)$, 
we need $D$-homothetic transformation  from a Sasakian metric $g$ with $\rho^T = 2 \omega^T$ to 
a Sasakian metric $g^\prime$ with $\rho^{\prime T} = (2m+2) \omega^{\prime T}$ because 
then $(C(S), \bar{g^\prime})$ is Calabi-Yau by Fact \ref{s3}, (b). This $D$-homothetic transformation is
achieved by $r^\prime = r^{1/(m+1)}$ where $r$ and $r^\prime$ are the radial functions on $C(S)$ for 
$g$ and $g^\prime$ respectively. Then 
$\tau(t)$ in \eqref{CY3} and 
$\varphi(\tau)$ in \eqref{CY1} respectively become
\begin{equation}\label{CY4}
\tau(t) = \left(r^{\prime 2(m+1)} +1\right)^\frac1{m+1} - 1,
\end{equation}
and 
\begin{equation}\label{CY2}
\varphi(\tau) = \frac 2{m+1} \left(\left(r^{\prime 2(m+1)} + 1\right)^{\frac1{m+1}} 
- \frac 1{\left(r^{^\prime 2(m+1)} + 1\right)^\frac m{m+1}}\right).
\end{equation}
Then by \eqref{C13},
\begin{equation}\label{CY5}
\omega_\varphi 
= \left(r^{\prime 2(m+1)} + 1\right)^\frac 1{m+1} (m+1)\omega^{\prime T}  + \frac1{\varphi(\tau)} i \p\tau \wedge \bp\tau.
\end{equation}
Comparing this with
\begin{equation}\label{CY6}
(m+1)\omega^\prime = \frac{(m+1)}2 i \p \bp r^{\prime 2} = r^{\prime 2} (m+1)\omega^{\prime T} 
+ 2(m+1)i\p r^\prime \wedge \bp r^\prime,
\end{equation}
one can check that the difference of \eqref{CY5} and \eqref{CY6} is $O(r^{\prime -2m})$.
In fact, the difference of the first terms is $O(r^{\prime -2m})$, and the second terms $O(r^{\prime -(2m+2)})$.
Since $\omega^\prime$ is Ricci-flat so is $(m+1)\omega^\prime$. 
Thus the Calabi-Yau metric obtained in Theorem \ref{main2} is asymptotic to the Calabi-Yau
cone metric $(m+1)\omega^\prime$ with order $O(r^{\prime -2m})$.
It follows from the uniqueness theorem of \cite{ConlonHein13} that our Calabi-Yau metric coincides
with those obtained in \cite{vanC10} and \cite{Goto}. 
\end{proof}

\section{Soliton analogues}
In this section we consider the case when the momentum construction \eqref{C13} on $C(S)$ 
satisfies the K\"ahler-Ricci soliton equation
\begin{equation}\label{KRS1}
\rho_\varphi -\alpha \omega_\varphi = -i\partial\barpartial Q(t)
\end{equation}
where $Q(t)$ is a smooth function of $t = \log r$ whose gradient is a holomorphic vector field.
Then by Lemma 4.1 in \cite{FutakiWang11}, $Q(t)$ is necessarily of the form
\begin{equation*}\label{KRS2}
Q = \mu\tau + c,
\end{equation*}
where $c$ is a constant. Using \eqref{C20}, one computes
\begin{eqnarray*}
i\partial\barpartial\, Q
&=& \frac{dQ}{d\tau} \varphi(\tau)dd^c t + \left(\frac{dQ}{d\tau}\varphi\right)'  \varphi dt \wedge d^ct \nonumber\\
&=& \mu\varphi(\tau)\omega^T + (\mu\varphi(\tau))' \varphi(\tau) dt \wedge d^ct.
\end{eqnarray*}
Comparing this with \eqref{C13} and \eqref{P2}, we obtain
\begin{equation}\label{KRS4}
\kappa - \frac{m\varphi + (1 + \tau)\varphi'}{1+\tau} = \alpha(1+\tau) - \mu\varphi(\tau)
\end{equation}
and
\begin{equation}\label{KRS5}
 - \left(\frac{m\varphi}{1+\tau} +  \varphi'\right)' = \alpha - (\mu\varphi(\tau))'.
\end{equation}
But \eqref{KRS5} follows from \eqref{KRS4}. 
If we require $\varphi(0) = 0$ and $\varphi^\prime(0) = 2$ it follows from \eqref{KRS4} that
\begin{equation*}\label{KRS6}
\kappa - 2 = \alpha,
\end{equation*}
and \eqref{KRS4} becomes
\begin{equation}\label{KRS7}
\varphi' + \left(\frac{m}{1+\tau} - \mu\right)\varphi + (\kappa - 2)\tau - 2 = 0.
\end{equation}
In general a solution to the ODE $y' + p(x)y = q(x)$ is given by
\begin{equation}\label{ode}
 y = e^{- \int p(x) dx} \left( \int q(x) e^{\int p(x) dx} dx +C\right) .
 \end{equation}
It follows from \eqref{ode} that the solution to \eqref{KRS7} is given by
\begin{equation}\label{KRS8}
\vph(\tau) = \frac{\nu e^{\mu(1+\tau)}}{(1+\tau)^m} + \frac{(\kappa-2)(1+\tau)}{\mu} +\frac{\kappa -2 - \frac{\kappa\mu}{m+1}}{\mu^{m+2}} \sum_{j=0}^m \frac{(m+1)!}{j!} \mu^j (1+\tau)^{j-m}
\end{equation}
for some constant $\nu$. But by the requirement $\varphi(0) = 0$, 
$\nu$ is determined by
\begin{eqnarray}\label{KRS9}
 \nu &=& e^{-\mu }  \left( \frac {-\kappa +2}\mu  + \frac{-\kappa + 2 +\frac{\kappa\mu}{m+1}}{\mu^{m+2}} \sum_{j=0}^m \frac{(m+1)!}{j!} \mu^j \right)\\
 &=:& \nu(\kappa,\mu).\nonumber
 \end{eqnarray}
Thus the solution \eqref{KRS8} becomes
\begin{eqnarray}\label{KRS10}
\vph(\tau) &=&  \left( \frac {-\kappa + 2}\mu  + \frac{-\kappa + 2 +\frac{\kappa\mu}{m+1}}{\mu^{m+2}} \sum_{j=0}^m \frac{(m+1)!}{j!} \mu^j \right)\frac{e^{\mu\tau}}{(1+\tau)^m} \nonumber\\
&&+ \frac{(\kappa - 2)(1+\tau)}{\mu} +\frac{\kappa - 2 - \frac{\kappa\mu}{m+1}}{\mu^{m+2}} \sum_{j=0}^m \frac{(m+1)!}{j!} \mu^j (1+\tau)^{j-m}.
\end{eqnarray}

 Let $M$ be a Fano manifold of dimension $m$, and $L \to M$ be a negative line bundle with $K_M = L^{p}$, $p \in \bfZ^+$.
 Take $k \in \bfZ^+$. Let $S$ be the $U(1)$-bundle associated with $L^{k}$, which is a regular Sasakian manifold with the K\"ahler cone $C(S)$ biholomorphic to $L^{k}$ minus
 the zero section. 
 We assume that
 $S$ admits a possibly irregular Sasakian $\eta$-Einstein metric which is $S^1$-bundle-adapted 
 in the sense of Definition \ref{bundle}. 
 When $M$ is toric this is indeed the case as we saw in the proof of Theorem \ref{main2}.

 Let $\kappa = \frac{2p}k$ and $\omega$ be the $\eta$-Einstein
 Sasakian metric such that
 $$ \rho^T = \kappa \omega^T,$$
 where $\omega^T$ and $\rho^T$ are respectively the transverse K\"ahler form and its transverse Ricci form. 
 In this set-up we start the momentum construction for the K\"ahler-Ricci soliton \eqref{KRS1}, and following
 the subsequent computations we obtain the solution \eqref{KRS10} requiring $\varphi(0) = 0$ and $\varphi'(0) = 2$.
 But we have not specified the region of the variable $\tau$ yet.
 
 First we consider the case $k \ge p$. Then of course $\kappa \le 2$ and $\alpha \le 0$.
 In this case we take $\mu < 0$, and take the region of the variable $\tau$ to be $[0, \infty)$.
 \begin{claim}\label{claim1}
 $\varphi > 0$ on $(0, \infty)$.
 \end{claim}
 \begin{proof}
 Since
 $\varphi(0) = 0$ and $\varphi'(0) = 2$ then $\varphi > 0$ on $(0,\epsilon)$ for some $\epsilon > 0$. 
 Suppose that $\varphi$ is non-positive somewhere on $(0, \infty)$, and that $a>0$ is the smallest point
 where 
 $\varphi(a) = 0$. Then $\varphi'(a) \le 0$. But
 by \eqref{KRS7}, $\varphi'(a) = 2 + (2-\kappa)a > 0$ for any $a \in (0, \infty)$. 
 This is a contradiction.
 \end{proof}
  \begin{claim}\label{claim2}
 $F'$ maps $(-\infty,\infty)$ diffeomorphically onto $(0, \infty)$. 
 \end{claim}
 \begin{proof}
 Since $F''(t) = \varphi(\tau) > 0$ by Claim \ref{claim1}, $F'$ maps its domain diffeomorphically onto its image.
 From $\varphi(0) = 0$, $\varphi'(0) = 2$ and \eqref{C8}, we see that $t \to -\infty$ as $\tau \to 0$. On the other hand,
 as we take $\mu < 0$, $\varphi(\tau)$ grows linearly as $\tau \to \infty$, and we see from \eqref{C8} that 
 $t \to \infty$ as $\tau \to \infty$. Hence the domain and the range of $F'$ are respectively $(-\infty,\infty)$ and $(0, \infty)$.
 \end{proof}
\begin{claim}\label{claim3}
The metric $\omega_\varphi$ defines a complete metric on $L^{\otimes k}$.
The resulting metric along the zero section has an expression which 
 can be extended to the total space of $L^k$, 
restricts to the associated unit circle bundle $S \cong \{r=1\}$ as a transversely K\"ahler-Einstein 
(Sasakian $\eta$-Einstein) metric scaled by a constant given by \eqref{C139} in the Reeb flow direction, and 
there is a Riemannian submersion from the scaled Sasakian $\eta$-Einstein metric to the
induced metric of the zero section.
 \end{claim}
 \begin{proof} Since $\varphi(\tau)$ grows linearly as $\tau \to \infty$, $\varphi(0) = 0$ and $\varphi'(0) = 2$
then this claim follows from the same arguments as in the proof of Proposition \ref{C30}. 
 \end{proof}
 
 Next we turn to the case $k < p$. 
 Then $\kappa = \frac{2p}k > 2$ and $\alpha = \kappa - 2 > 0$.
\begin{claim}\label{claim4}
In \eqref{KRS9} we can take some positive $\mu$ so that $\nu(\kappa,\mu) = 0$, and 
with this choice of $\mu$, the solution $\varphi(\tau)$ is expressed as
\begin{equation}\label{KRS11}
\vph(\tau) = \frac{(\kappa - 2)(1+\tau)}{\mu} +\frac{\kappa - 2 - \frac{\kappa\mu}{m+1}}{\mu^{m+2}} \sum_{j=0}^m \frac{(m+1)!}{j!} \mu^j (1+\tau)^{j-m}.
\end{equation}
 \end{claim}
 \begin{proof}
 The leading order term as $\mu \to 0$ is $(-\kappa + 2)/\mu^{m+2}$, and thus $\nu(\kappa,\mu) \to -\infty$.
 On the other hand the leading order term in the paretheses as $\mu \to \infty$ is $2/\mu$, and thus $\nu(\kappa,\mu)$ tends to be positive.
 Thus we can find a positive $\mu$ such that $\nu(\kappa,\mu) = 0$. (This choice of $\mu$ is actually unique
 since the coefficients of the monomials inside $\nu(\kappa,\mu)$ changes sign only once when arranged
 from lower to higher. This argument is due to \cite{FIK}.)
 \end{proof}
 \begin{claim}\label{claim5}
$\vph(\tau) > 0$ 
for all $\tau > 0$.
 \end{claim}
 \begin{proof}
 Since $\varphi(0) = 0$ and $\frac{\kappa -2 }{\mu} >0$ we must have 
$$\frac{\kappa - 2 - \frac{\kappa\mu}{m+1}}{\mu^{m+2}} < 0.$$
As $\varphi'(0) > 0$ we have $\varphi(\tau) > 0$ for small $\tau > 0$. Then as $\tau$ gets bigger 
the right hand side \eqref{KRS11} gets bigger because of the signs of the coefficients of the first and
the second term.
 \end{proof}
\begin{claim}\label{claim6}
 $F'$ maps $(-\infty,\infty)$ diffeomorphically onto $(0,\infty)$. 
 \end{claim}
 \begin{proof}
 Since $F''(t) = \varphi(\tau) > 0$ by Claim \ref{claim5}, $F'$ maps its domain diffeomorphically onto its image.
 From $\varphi(0) = 0$, $\varphi'(0) = 2$ and \eqref{C8}, we see that $t \to -\infty$ as $\tau \to 0$. On the other hand,
 since $\kappa > 2$ and $\mu > 0$, $\varphi(\tau)$ grows linearly as $\tau \to \infty$, and we see from \eqref{C8} that 
 $t \to \infty$ as $\tau \to \infty$. Hence the domain and the range of $F'$ are respectively $(-\infty,\infty)$ and $(0, \infty)$.
 \end{proof}
\begin{claim}\label{claim7}
The metric $\omega_\varphi$ defines a complete metric on $L^{k}$.
The resulting metric along the zero section has an expression which 
 can be extended to the total space of $L^k$, 
restricts to the associated unit circle bundle $S \cong \{r=1\}$ as a transversely K\"ahler-Einstein 
(Sasakian $\eta$-Einstein) metric scaled by a constant given by \eqref{C139} in the Reeb flow direction, and 
there is a Riemannian submersion from the scaled Sasakian $\eta$-Einstein metric to the
induced metric of the zero section. \end{claim}
 \begin{proof} Since $\varphi(\tau)$ grows linearly as $\tau \to \infty$, $\varphi(0) = 0$ and $\varphi'(0) = 2$
 this claim follows from the same arguments as in the proof of Proposition \ref{C30}. 
 \end{proof}
 
 \begin{proof}[Proof of Theorem \ref{main4}.]
Let $M$ be a toric Fano manifold of dimension $m$, and $L \to M$ be a negative line bundle with $K_M = L^{p}$, $p \in \bfZ^+$.
Take $k \in \bfZ^+$. By the same arguments using $D$-homothetic transformations one can show that the $U(1)$-bundle $S$ associated with $L^{k}$ 
admits a possibly irregular Sasakian $\eta$-Einstein metric which is $S^1$-bundle-adapted 
 in the sense of Definition \ref{bundle}. 
 Then by Claim \ref{claim1}, Claim \ref{claim2} and Claim \ref{claim3}, the total space of $L^{k}$ admits a complete expanding K\"ahler-Ricci soliton
 if $k > p$, and a complete steady K\"ahler-Ricci soliton if k = p,.
 Further by Claim \ref{claim4}, Claim \ref{claim5}, Claim \ref{claim6} and Claim \ref{claim7}, 
the total space of $L^{k}$ admits a complete shrinking K\"ahler-Ricci soliton if $k<p$.
As stated in Claim \ref{claim3} and Claim \ref{claim7}, the resulting metric in each case 
along the zero section has an expression which 
 can be extended to the total space of $L^k$, 
restricts to the associated unit circle bundle $S \cong \{r=1\}$ as a transversely K\"ahler-Einstein 
(Sasakian $\eta$-Einstein) metric scaled by a constant given by \eqref{C139} in the Reeb flow direction, and 
there is a Riemannian submersion from the scaled Sasakian $\eta$-Einstein metric to the
induced metric of the zero section.
 \end{proof}

\end{document}